\documentclass[12pt,leqno]{article}
\usepackage{graphicx}
\usepackage{epstopdf,epsfig}
\usepackage[colorlinks=false]{hyperref}
\usepackage{csquotes}

\usepackage{xcolor}

\usepackage{amsmath,amsthm,amsfonts,amssymb,latexsym,amscd,enumerate}

\usepackage{palatino}
\usepackage[mathcal]{euler}

\usepackage{xy}
\xyoption{all}

\swapnumbers

\theoremstyle{plain}
\newtheorem*{theorem*}{Theorem} 
\newtheorem*{lemma*}{Lemma}
\newtheorem*{assumption*}{Assumption}

\newtheorem{theorem}{Theorem} 
\newtheorem{lemma}[theorem]{Lemma}
\newtheorem{corollary}[theorem]{Corollary}
\newtheorem*{corollary*}{Corollary}
\newtheorem{proposition}[theorem]{Proposition}
\newtheorem{conjecture}[theorem]{Conjecture}
\theoremstyle{definition}
\newtheorem*{definition*}{Definition}
\newtheorem{definition}[theorem]{Definition}
\newtheorem{example}[theorem]{Example}

\newtheorem{remark}[theorem]{Remark}

\newtheorem*{observation*}{Observation}

\theoremstyle{remark}

\numberwithin{theorem}{section}

\numberwithin{equation}{section}

\newcommand{\Calg}{C^\ast\text{-algebra}}
\newcommand{\Calgs}{C^\ast\text{-algebras}}

\begin{document} 

\title{Direct Splitting Method for\\ the Baum--Connes Conjecture} 
\author{Shintaro Nishikawa\footnote{
Department of Mathematics, Penn State University, University Park, PA 16802, USA.
}}

\date{\today}

\maketitle

\begin{abstract}
We develop a new method for studying the Baum--Connes conjecture, which we call the direct splitting method. We introduce what we call property $(\gamma)$ for $G$-equivariant Kasparov cycles. We show that the existence of a $G$-equivariant Kasparov cycle with property $(\gamma)$ implies the split-injectivity of the assembly map $\mu^G_ A$ for any separable $G$-$\Calg$ $A$. We also show that if such a cycle exists, the assembly map $\mu^G_ A$ is an isomorphism if and only if the cycle acts as the identity on the right-hand side group $K_\ast(A\rtimes_rG)$ of the Baum--Connes conjecture. In a separate paper, with J. Brodzki, E. Guentner and N. Higson, we use this method to give a finite-dimensional proof of the Baum--Connes conjecture for groups which act properly and co-compactly on a finite-dimensional CAT(0)-cubical space.
\end{abstract}


\section{Introduction}

Let $G$ be a second countable, locally compact, Hausdorff topological group. In this paper, we assume that there is a co-compact model $\underline{E}G$ of the universal proper $G$-space (see \cite[Definition 1.6]{BCH94}).

In 1994, Baum, Connes and Higson gave the current formulation of the Baum-Connes conjecture \cite{BCH94} using Kasparov's equivariant $KK$-theory \cite{Ka88}. The Baum--Connes conjecture with coefficients for $G$ is the assertion that the Baum-Connes assembly map
\[
\mu_A^G \colon KK^G_\ast(C_0(\underline{E}G), A) \to KK_\ast(\mathbb{C}, A\rtimes_rG) 
\]
is an isomorphism for any separable $G$-$\Calg$ $A$, where $A\rtimes_rG$ denotes the reduced crossed product of $A$ by $G$. 

The assembly map $\mu^G_A$ is proved to be an isomorphism for a reasonably large class of groups. Moreover, it is proved to be split-injective for a much larger class. We list some of the important known results:
 \begin{itemize}
\item $\mu^G_A$ is an isomorphism for all amenable, or more generally, all a-T-menable groups $G$, and all $A$ (Higson and Kasparov \cite{HK97}, \cite{HK01}). \\
\item $\mu^G_A$ is an isomorphism for all word-hyperbolic groups $G$, and all $A$ (Lafforgue \cite{La12}). \\
\item $\mu^G_A$ is split-injective for all discrete groups $G$ which coarsely embed into a Hilbert space, and all $A$ (Skandalis, Tu and Yu \cite{STY02}). 
\end{itemize}

Most of the known results (including the results listed above) rely on the so-called $\gamma$-element method, also called the dual Dirac method. The main step, initially developed and used by Kasparov, is to find an element $\gamma$ in the Kasparov ring $R(G)=KK^G(\mathbb{C}, \mathbb{C})$ which is uniquely characterized by the following two properties (see \cite{Ka88} or \cite{Tu99} for details):
\begin{enumerate}[(i)]
\item $\gamma=1_K$ in $R(K)$ for any compact subgroup $K$ of $G$.
\item  $\gamma$ factors through a proper $G$-$\Calg$. That is, there is a separable, proper $G$-$C^\ast$-algebra $P$ so that $\gamma$ is the Kasparov product $\delta\otimes_P d$ of two elements $d$ (called Dirac element) in $KK^G(P, \mathbb{C})$ and $\delta$ in $KK^G(\mathbb{C}, P)$ (called dual Dirac element).   
\end{enumerate}

The mere existence of $\gamma$ has the following immediate consequences:
\begin{enumerate}[(i)]
\item $\mu^G_A$ is split-injective for any $A$ (the strong Novikov conjecture). \\
\item $\mu^G_A$ is an isomorphism if and only if $\gamma$ acts as the identity on the right-hand side group $KK_\ast(\mathbb{C}, A\rtimes_rG)$  via the composition
\begin{equation} \label{eq_comp} KK^G_\ast(\mathbb{C}, \mathbb{C}) \to KK^G_\ast(A, A) \to KK_\ast(A\rtimes_rG, A\rtimes_rG). 
\end{equation}
\end{enumerate}

All the following classes of groups are known to admit a $\gamma$-element.
 \begin{itemize}
\item All groups $G$ that act properly and isometrically on a simply connected, complete Riemannian manifold with non-positive sectional curvature, or more generally on a special manifold (Kasparov \cite{Ka88}). This class contains all closed subgroups of Lie groups, or more generally, of almost connected groups. \\
\item All groups $G$ that act properly on a Euclidean building (Kasparov and Skandalis \cite{KS91}). This class contains all closed subgroups of $GL(K)$ where $K$ is a non-archimedean local field. \\
\item All groups $G$ that act properly and isometrically on a bolic space  (Kasparov and Skandalis \cite{KS03}). This class contains all word-hyperbolic groups. \\
\item All groups $G$ that act metrically properly and isometrically on a Hilbert space (Higson and Kasparov \cite{HK97}, \cite{HK01}). This class contains all amenable groups.  \\
\item All discrete groups $G$ that coarsely embed into a Hilbert space (Tu \cite{Tu05}). This class contains all discrete exact groups (c.f. \cite{Oza00}, \cite{GK02}).  \\
\end{itemize}

There is no doubt that the $\gamma$-element method (the dual Dirac method) is a versatile and powerful approach to the Baum--Connes conjecture. However, there can be difficulties in applying the $\gamma$-element method, most notably difficulties in constructing a suitable proper algebra $P$ and factorization. For example, the case of Euclidean buildings, studied by Kasparov and Skandalis \cite{KS91} required quite ingenious constructions, and the case of CAT(0) cubical spaces is still more challenging in spite of the fact that there is a very natural candidate for the $\gamma$-element \cite{BGH19}.

In this article, we shall introduce a more flexible alternative to the $\gamma$-element method.  Let $(H, T)$ be a $G$-equivariant Kasparov cycle which defines the element $[H, T]$ in the Kasparov ring $R(G)=KK^G(\mathbb{C} ,\mathbb{C})$: the Hilbert space $H$ is equipped with a grading and a unitary representation of $G$; the odd, bounded self-adjoint operator $T$ is such that $1-T^2$ and $g(T)-T$ are compact operators for any $g$ in $G$. We say that {\bf the cycle $(H, T)$ has property $(\gamma)$} if the following are satisfied (see Definition \ref{def_gamma'}):
\begin{enumerate}[(i)]
\item $[H, T]=1_K$ in $R(K)$ for any compact subgroup $K$ of $G$.
\item  There is a non-degenerate, $G$-equivariant representation of the $G$-$\Calg$ $C_0(\underline{E}G)$ on $H$ so that the following holds.
\end{enumerate}
\begin{enumerate}
\item[(ii.i)] the function $(g \mapsto [g(\phi), T])$ is in $C_0(G, K(H))$ for any $\phi \in C_0(\underline{E}G)$.
\item[(ii.ii)] $\int_{g \in G} g(c)T g(c)d{\mu_G}(g) - T \in K(H)$ for some cut-off function $c$ on $\underline{E}G$.
\end{enumerate}

For any cycle $(H, T)$ with property $(\gamma)$, we shall define a natural map
\[
\nu_A^{G,T}\colon KK_\ast(\mathbb{C}, A\rtimes_rG) \to KK^G_\ast(C_0(\underline{E}G), A)
\]
which we call the $(\gamma)$-morphism defined by $(H, T)$ (see Definition \ref{def_gamma'_mor}). The following are our main results. Notice the analogy between them and the consequences of the existence of the gamma element described above.

\begin{theorem*} \label{thm_BCC} (See Theorem \ref{thm_BCCmain}) Suppose there is a $G$-equivariant Kasparov cycle $(H, T)$ with property $(\gamma)$. Then:
\begin{enumerate}[(i)]
\item the strong Novikov conjecture holds for $G$, i.e. the assembly map $\mu^G_A$ is split-injective for any $A$,
\item the $(\gamma)$-morphism $\nu^{G, T}_A$ is a left-inverse of the assembly map $\mu^G_A$,
\item the assembly map $\mu^G_A$ is an isomorphism if and only if the cycle $(H, T)$ acts as the identity on the right-hand side group $KK_\ast(\mathbb{C}, A\rtimes_rG)$  via the composition \eqref{eq_comp}.
\end{enumerate}
\end{theorem*}

\begin{corollary*}  \label{cor_BCC} (See Corollary \ref{cor_BCCmain}) Suppose there is a $G$-equivariant Kasparov cycle $(H, T)$ with property $(\gamma)$ which acts as the identity on the right-hand side group $KK_\ast(\mathbb{C}, A\rtimes_rG)$ via the composition \eqref{eq_comp2} for any $A$. Then, the Baum--Connes conjecture with coefficients holds for $G$. The $(\gamma)$-morphism $\nu^{G, T}_A$ is the inverse of the assembly map $\mu^G_A$ for any $A$.
\end{corollary*} 

\begin{corollary*}  \label{cor_BCC2} (See Corollary \ref{cor_BCCmain2}) Suppose there is a $G$-equivariant Kasparov cycle $(H, T)$ with property $(\gamma)$ which is homotopic to $1_G$. Then, the Baum--Connes conjecture with coefficients holds for $G$. The $(\gamma)$-morphism $\nu^{G, T}_A$ is the inverse of the assembly map $\mu^G_A$ for any $A$.
\end{corollary*}

By design, the definition of property $(\gamma)$ aims to copy features of an actual $\gamma$-element. Hence, the following results are perhaps not surprising:

\begin{theorem*} (See Theorem \ref{thm_implies}) If the $\gamma$-element exists for some group $G$, then $\gamma$ is represented by a $G$-equivariant Kasparov cycle with property $(\gamma)$.
\end{theorem*}

\begin{theorem*}(See Theorem \ref{thm_unique}) A $G$-equivariant Kasparov cycle $(H, T)$ with property $(\gamma)$, if exists, is unique up to homotopy.
\end{theorem*}

\begin{definition*} (See Definition \ref{def_gamma_element}) Suppose there is a $G$-equivariant Kasparov cycle $(H, T)$ with property $(\gamma)$. We define the $(\gamma)$-element for $G$ to be the unique element in $R(G)$ which is represented by a $G$-equivariant  Kasparov cycle $(H, T)$ with property $(\gamma)$. 
\end{definition*}

If the $\gamma$-element exists for a group $G$, the two notions of $\gamma$-element and $(\gamma)$-element coincide, i.e. $\gamma$ is the $(\gamma)$-element.

\begin{theorem*}(See Theorem \ref{thm_idempotent}) The $(\gamma)$-element is an idempotent in the Kasparov ring $R(G)$.
\end{theorem*}

In \cite{BGHN19} with J. Brodzki, E. Guentner and N. Higson, we use our main result to give a new, finite-dimensional proof of the Baum--Connes conjecture for groups which act properly and co-compactly on a finite-dimensional CAT(0) cubical space. The main idea of this work is to modify the natural candidate for the $\gamma$-element constructed in \cite{BGH19} to make it have property $(\gamma)$. In this way, we obtain a proof of the Baum--Connes conjecture which avoids dealing with the difficulties in constructing a suitable algebra $P$ and factorization.

It is an interesting challenge to find the correct definition of property $(\gamma)$ for general groups which may not admit a co-compact model of $\underline{E}G$. This will be investigated in the future.

In \cite{NP19}, with V. Proietti, the idea of the ($\gamma$)-morphism is used to prove a strong duality result between the reduced  group $C^\ast$-algebra $C^\ast_r(G)$ and the crossed product algebra $C_0(\underbar{E}G)\rtimes_rG$ for any discrete group $G$ which admits a $G$-equivariant Kasparov cycle $x$ with property $(\gamma)$ such that $j^G_r(x)=\mathrm{id}_{C^\ast_r(G)}$ where $j^G_r$ is the descent map.

\paragraph*{Acknowledgements} I would like to thank my advisor, Nigel Higson at Penn State University for many of his insights on this project. This work was born out of other project with Jacek Brodzki, Erik  Guentner and Nigel Higson ``On the Baum--Connes conjecture and CAT(0)-cubical spaces'' \cite{BGHN19}. I would like to thank all of them for kindly welcoming me to joint the project. This paper is partly written during the summer 2018 at Penn State University under the support of Professor Higson. I would like to thank him for the support during this time. I would also like to thank Heath Emerson, Siegfried Echterhoff, Kang Li, Yiannis Loizides, Wolfgang L\"uck, Valerio Proietti, Alain Valette and Rufus Willett for their helpful comments.
 
 
\section{Property $(\gamma)$}
Throughout this article, $G$ denotes a second countable, locally compact, Hausdorff topological group which admits a co-compact model $E=\underline{E}G$ of the universal proper $G$-space. For a discrete, torsion-free group $G$, this means that there is a compact model $BG$ of the classifying space of $G$ and that $E$ is its universal cover, equipped with the deck-transformation action by $G$. The following are some examples of $G$ (and $E)$ we can take:
\begin{itemize}
\item the free abelian group $\mathbb{Z}^n$ and for $E$, the Euclidean space $\mathbb{R}^n$ equipped with the translation action by $\mathbb{Z}^n$.
\item the free group $F_n$ with $n$ generators and for $E$, the universal cover of the wedge $S^1\vee \cdots \vee S^1$ of $n$-many circles.
\item the fundamental group of a compact, aspherical manifold and for $E$, the universal cover of the manifold. 
\item any co-compact closed subgroup of a connected, semi-simple Lie group $L$ (more generally, of any almost connected, locally compact topological group $L$) and for $E$, the homogeneous space $L/K$ where $K$ is a maximal compact subgroup of $L$.
\end{itemize}
We recall the definition of a $G$-equivariant Kasparov cycle from \cite{Ka88}. For a Hilbert space $H$, we denote by $K(H)$, the algebra of all compact operators on $H$.

\begin{definition} A $G$-equivariant Kasparov cycle is a pair $(H, T)$ where $H$ is a separable, graded $G$-Hilbert space and $T$ is an odd, bounded and self-adjoint operator on $H$ with the following two conditions. 
\[
\text{Fredholm condition:} \,\,\,\,\, 1-T^2 \in K(H).
\]
\[
\text{Almost-$G$-equivariance:} \,\,\,\,  g(T)-T \in K(H) \,\,\, \text{for any $g \in G$},
\]
where $g(T) = u_g T  u_g^\ast$ is the conjugation by the representation $u_g$ of $g$ in $G$.
\end{definition}

There is a natural notion of homotopies of $G$-equivariant Kasparov cycles (see \cite[Definition 2.3]{Ka88}). The commutative ring $R(G)=KK^G(\mathbb{C}, \mathbb{C})$ is defined as the set of homotopy equivalence classes of $G$-equivariant Kasparov cycles. We write by $[H, T]$, the element in $R(G)$ defined by a $G$-equivariant Kasparov cycle $(H, T)$. The addition and the multiplication of the ring $R(G)$ are defined by the direct sum operation and by the Kasparov product. See \cite{Ka88}, \cite{Bl98} for more details.

Any pair ($H^{(0)}, H^{(1)}$) of finite-dimensional unitary representations of the group $G$ defines a $G$-equivariant Kasparov cycle $(H^{(0)}\oplus H^{(1)}, 0)$ and hence an element in $R(G)$. We denote by $1_G$, the one $[\mathbb{C}\oplus0, 0]$ which corresponds to the trivial representation of $G$. The element $1_G$ is the unit in the ring $R(G)$.

As explained briefly in the introduction, there is the notion called the gamma element $\gamma$ in the ring $R(G)$ which plays a central role in the Baum--Connes conjecture. The mere existence of the gamma element for the group $G$ not only has profound consequences for the Baum--Connes conjecture (for example, it implies the split-injectivity of the assembly maps $\mu^G_A$ for all coefficients $A$) but it also has several applications in $C^\ast$-algebra theory. The next definition, property $(\gamma)$ for a $G$-equivariant Kasparov cycle, is supposed to capture the remarkable properties of the gamma element at the level of $G$-equivariant Kasparov cycles. 

For a locally compact, Hausdorff space $X$ and a $\Calg$ $A$, we denote by $C_0(X, A)$, the $\Calg$ of continuous $A$-valued functions on $X$ vanishing at infinity. Recall that a compactly supported, continuous, non-negative function $c$ on a (co-compact) proper $G$-space $X$ such that the Haar integral 
\[\displaystyle \int_{g \in G}g(c)^2d\mu_G(g)=1
\] is called a cut-off function. We also note that for any closed subgroup $K$ of $G$, a $G$-equivariant Kasparov cycle can be viewed as a $K$-equivariant cycle. 

\begin{definition} \label{def_gamma'} We say that a $G$-equivariant Kasparov cycle $(H, T)$ has property $(\gamma)$ if the following are satisfied 
\begin{enumerate}[(i)]
\item $[H, T]=1_K$ in $R(K)$ for any compact subgroup $K$ of $G$.
\item  There is a non-degenerate, $G$-equivariant representation of the $G$-$\Calg$ $C_0(E)$ on $H$ so that the following holds.
\end{enumerate}
\begin{enumerate}
\item[(ii.i)] the function $(g \mapsto [g(\phi), T])$ is in $C_0(G, K(H))$ for any $\phi \in C_0(E)$.
\item[(ii.ii)] $\int_{g \in G} g(c)T g(c)d{\mu_G}(g) - T \in K(H)$ for some cut-off function $c$ on $E$.
\end{enumerate}
\end{definition}

We first show that the gamma element, if exists, is represented by a $G$-equivariant Kasparov cycle with property $(\gamma)$. Before proving this, let us first prove simple and useful lemmas some of which we shall use in later sections as well.

\begin{lemma} \label{lem_nondegen} Let $X$ be a locally compact, Hausdorff, proper $G$-space. For any non-degenerate representation of the $G$-$\Calg$ $C_0(X)$ on a $G$-Hilbert space $H$, we have
\[
\lim_{g \to \infty}||\phi g(T)|| = 0
\]
for any compact operator $T$ on $H$ and for any $\phi$ in $C_0(X)$.
\end{lemma}
\begin{proof} It is enough to show that
\[
\lim_{g \to \infty}||g(\phi) T|| = 0
\]
for any rank one operator $T$ on $H$. This follows from
\[
\lim_{g \to \infty}||g(\phi) v|| = 0
\]
for any vector $v$ in $H$. Since $C_0(X)$ is non-degenerately represented on $H$, to show the last claim, we just need to show the claim for vectors of the form $v=\phi_0v_0$ for some $\phi_0$ in $C_0(X)$ and $v_0$ in $H$. The claim now follows from
\[
\lim_{g \to \infty}||g(\phi) \phi_0|| = 0.
\]
This holds since $X$ is a proper $G$-space. Indeed, if we assume $\phi$ and $\phi_0$ are compactly supported,  
\[
g(\phi) \phi_0 = 0
\]
for $g$ outside a sufficiently large compact subset of $G$.
\end{proof}

\begin{lemma} \label{lem_sum}  Let $H$ be a Hilbert space with a non-degenerate representation of the $\Calg$ $C_0(X)$ for a locally compact, Hausdorff space $X$. Let $\rho_j, \rho'_j$ be two families of real-valued functions indexed over $j$ in $J$ such that $\rho_j^2$, $\rho'^2_j$ are summable over $J$ in the strong operator topology on $H$. Then, for any uniformly bounded family of operators $T_j$ for $j$ in $J$, the sum
\[
\sum_{j\in J}\rho'_j T_j \rho_j
\]
converges in the strong operator topology on $H$ and defines an operator on $H$ with norm-bound $C\displaystyle\sup_{j\in J}||T_j||$ for some constant $C>0$ which is independent of the family $T_j$.
\end{lemma}
\begin{proof} The square-summability of $\rho_j, \rho'_j$ guarantees the following map defines a bounded operator $V$ from $H$ to $\displaystyle\bigoplus_{j\in J}H$:
\[
V\colon \xi \mapsto (\rho_j\xi)_{j\in J} \,\,\, \text{with adjoint} \,\,\, V^\ast\colon (\xi_j)_{j\in J} \to \sum_{j\in J}\rho_j\xi_j,
\]
and similarly $V'$ using $\rho_j'$. The assertion is immediate by observing
\[
\sum_{j\in J}\rho'_j T_j \rho_j = V'^\ast(T_j)_{j\in J} V.
\]
\end{proof}

\begin{lemma}\label{lem_nondegen_int}  Let $H$ be a $G$-Hilbert space with a non-degenerate representation of the $G$-$\Calg$ $C_0(X)$ for a locally compact, Hausdorff, proper $G$-space $X$. Let $\phi_0$, $\phi_1$ be compactly supported real-valued functions on $X$. Let $(T_g)_{g\in G}$ be a uniformly bounded family of operators on $H$ which defines a bounded operator on $L^2(G, H)$. Then, the map
\[
v \mapsto \int_{G}g(\phi_0)T_gg(\phi_1)v d\mu_G(g)
\]
on $H$ define the bounded operator $\int_{G}g(\phi_0)T_gg(\phi_1) d\mu_G(g)$ on $H$, and we have
\[
||\int_{G}g(\phi_0)T_gg(\phi_1) d\mu_G(g)|| \leq C\sup_{g\in G}||T_g||
\]
for some constant $C$ which only depends on $\phi_0$, $\phi_1$.
\end{lemma}
\begin{proof} This can be shown quite analogously to the proof of Lemma \ref{lem_sum}. This time, we use the following bounded operator $V$ from $H$ to $L^2(G, H)$:
\[
V\colon \xi \mapsto (g(\phi_1)\xi)_{g\in G} \,\,\, \text{with adjoint} \,\,\, V^\ast\colon (\xi_g)_{g\in G} \to \int_{G}g(\phi_1)\xi_gd\mu_G(g),
\]
and similarly $V'$ using $\phi_0$.  The assertion is immediate by observing
\[
\int_{G}g(\phi_0)T_gg(\phi_1) d\mu_G(g) = V'^\ast(T_g)_{g\in G} V.
\]
\end{proof}

\begin{lemma}\label{lem_comp_int} Let $H$ be a $G$-Hilbert space with a non-degenerate representation of the $G$-$\Calg$ $C_0(X)$ for a locally compact, Hausdorff, proper $G$-space $X$. For any compactly-supported functions $\phi_0$ and $\phi_1$ on $X$ and a compact operator $T$, the operator $
\int_{g \in G}g(\phi_0)Tg(\phi_1)d\mu_G(g)$ on $H$ is a compact operator with norm-bound $C||T||$ for some constant $C>0$ which is independent of $T$.
\end{lemma}
\begin{proof} Take a compactly-supported function $\chi$ on $X$ taking value $1$ on the support of $\phi_1$. Rewrite the operator as
\[
\int_{g \in G}g(\phi_0)Tg(\phi_1)d\mu_G(g) = \int_{g \in G}g(\phi_0)(Tg(\chi))g(\phi_1)d\mu_G(g).
\]
Combined with the observation that the norm $||Tg(\chi)||$ vanishes as $g$ goes to infinity (see Lemma \ref{lem_nondegen}), the claim follows from Lemma \ref{lem_nondegen_int} and that the operator 
$\int_{g \in K}g(\phi_0)Tg(\phi_1)d\mu_G(g)$ is a compact operator for any compact subset $K$ of $G$.
\end{proof}

\begin{lemma} \label{lem_property_alg} Let $H$ be a $G$-Hilbert space with a non-degenerate representation of the $G$-$\Calg$ $C_0(X)$ for a locally compact, Hausdorff, co-compact, proper $G$-space $X$. For a cut-off function $c$ on $X$, the operators $S$ on $H$ satisfying 
\[
(g \mapsto [g(\phi), S]) \in C_0(G, K(H)) \,\,\, \text{for any $\phi$ in $C_0(X)$},
\]
\[
\int_{g \in G} g(c)S g(c)d{\mu_G}(g) - S \in K(H),
\]
form a $G$-$C^\ast$-algebra of operators on $H$ containing the compact operators on $H$.
\end{lemma}
\begin{proof} That the operators on $H$ satisfying the first condition form a $G$-$C^\ast$-algebra is quite straightforward. That all compact operators on $H$ satisfies the two conditions follows from Lemma \ref{lem_nondegen}, \ref{lem_comp_int}. That the second condition is closed under addition, adjoint and translation by an element in $G$ is also immediate. Let $S_1$, $S_2$ be operators satisfying the two conditions. We have
\[
\left(\int_{g \in G} g(c)S_1 g(c)d{\mu_G}(g)\right)\left(\int_{g \in G} g(c)S_2 g(c)d{\mu_G}(g)\right) - S_1S_2 \in K(H).
\]
Let $G_0$ be a compact subset of $G$ so that $h(c)c=0$ unless $h\in G_0$. The first term in the right above expression is equal to 
\[
\int_{(g, g') \in G\times G_0} g(c)S_1 g(c)(gg')(c)S_2 (gg')(c)d{\mu_{G\times G}}.
\]
Similarly to the proof of Lemma \ref{lem_nondegen_int}, \ref{lem_comp_int}, we can show that the difference between this term and
\[
\int_{(g, g') \in G\times G_0} g(c)S_1 S_2g(c)(gg')(c)(gg')(c)d{\mu_{G\times G}} = \int_{g \in G} g(c)S_1 S_2g(c)d{\mu_G}(g)
\]
is compact, showing that the second condition holds for the product $S_1S_2$. Lastly, the second condition is closed under norm-limit since the map $S\mapsto \int_{g \in G} g(c)S g(c)d{\mu_G}(g)$ is continuous in norm, which follows from Lemma \ref{lem_nondegen_int}.
\end{proof}
 
For $\Calgs$ $J$ and $A$, we say that $A$ is non-degenerately represented on $J$ if there is a $\ast$-homomorphism from $A$ to the multiplier algebra $M(J)$ of $J$ such that $AJ$ is dense in $J$. Lemma \ref{lem_nondegen} to Lemma \ref{lem_property_alg} have obvious generalizations by replacing a $G$-Hilbert space $H$ to a $G$-Hilbert $C^\ast$-module. A $G$-$\Calg$ $P$ is a proper $G$-$\Calg$ if there is a non-degenerate $G$-equivariant $\ast$-homomorphism from $C_0(E)$ to the center $Z(M(P))$ of the multiplier algebra $M(P)$ of $P$. For a proper $G$-$\Calg$ $P$, we write $P_c=C_c(E)P$, the dense subalgebra of compactly supported elements. For a subset $\Delta$ and a subalgebra $A$ in an algebra $B$, we say that $\Delta$ derives $A$ if $[\Delta, A]\subset A$. The following is a modified version of Kasparov's Technical Theorem (see \cite{Hig87}, \cite{Ka88}) (all algebras and commutators are graded).

\begin{theorem} \label{thm_tech}Let $J$ be a $\sigma$-unital $G$-$\Calg$ and $P$ be a proper $G$-$\Calg$ which is $G$-equivariantly and non-degenerately represented on $J$. Let $A_1$ and $A_2$ be $\sigma$-unital subalgebras in $M(J)$ such that $A_1$ is a $G$-subalgebra containing $P$ and $A_2$ is an arbitrary subalgebra (without the group action). Assume that $PA_1$ is dense in $A_1$. Let $\Delta, \Delta'$ be subsets in $M(J)$ which are separable in the norm topology such that $\Delta$ derives $A_1$ and $[\Delta', A_1]\subset J$. Assume that $A_1\cdot A_2\subset J$. Let $c$ be a cut-off function on $E$. Then, there are positive elements $M_1, M_2$ in $M(J)$ of degree $0$ such that $M_1 + M_2 =1$ and for $i=1,2$,
\begin{enumerate}[(I)]
\item $g(M_i)-M_i \in J$ for any $g$ in $G$,
\item $M_1x \in J$ for any $x$ in $A_1$,
\item $M_2x \in J$ for any $x$ in $A_2$,
\item $[M_i, x] \in J$ for any $x$ in $\Delta$,
\item the function $(g \mapsto [M_i, g(\phi)])$ belongs to $C_0(G, J)$ for any $\phi$ in $C_0(E)$,
\item the function $(g \mapsto M_2[x, g(\phi)])$ belongs to $C_0(G, J)$ for any $x$ in $\Delta'$ and for any $\phi$ in $C_0(E)$,
\item $\int_{g \in G} g(c)M_i g(c)d{\mu_G}(g) - M_i \in J$.
\end{enumerate}
\end{theorem}
\begin{proof}  We set $Y$ to be a compact subset of $C_0(E)$ consisting of functions $\phi$ such that the support $\{\phi\neq 0\}$ of $\phi$ is contained in the support $\{c\neq 0\}$ of $c$ and such that the translate $GY$ by $G$ generates $C_0(E)$. We note that such $Y$ exists and also that in practice, we can take $Y$ to be $\{c\}$ (see Remark \ref{rem_below}). Let $\tilde K_n$ be increasing and exhausting compact subsets of $G$, $\tilde\Delta$ (resp. $\tilde\Delta'$)  be a compact subset of $\mathrm{Span}\Delta$ (resp. of $\mathrm{Span}\Delta'$) whose span is dense in $\mathrm{Span}\Delta$  (resp. in $\mathrm{Span}\Delta'$) and $h_1, h_2$ be contractive, strictly positive elements of $A_1$ and $A_2$. First, using Lemma 1.4 of \cite{Ka88} (see also \cite{Hig87}), we take an approximate unit $a_n$ in $P_c$ such that the following conditions hold for $d_{n}=(a_n-a_{n-1})^\frac12$ where we set $a_0=0$:
\begin{enumerate}[(1)]
\item $||g(d_n)-d_n|| \leq 2^{-n}$ for any $g$ in $\tilde K_n$, 
\item $||d_nh_1|| \leq 2^{-n}$,
\item $||[d_n, x]||\leq 2^{-n}$  for any $x$ in $\tilde\Delta$.
\end{enumerate}
Let $K_n$ be increasing and exhausting compact subsets of $G$ so that 
\[
a_nh(c)=0 \,\,\, \text{unless $h\in K_n$}.
\]
Note, we have
\begin{enumerate}[(4)]
\item $d_nh(c)=0, \,\, d_nh(\phi)=0$ unless $h\in K_{n}$ for any $\phi \in Y$.
\end{enumerate}
Secondly, using Lemma 1.4 of \cite{Ka88} again, we take an approximate unit $u_n$ in $J$ such that the following conditions hold:
\begin{enumerate}[(a)]
\item $||g(u_n)-u_n|| \leq 2^{-n}$ for any $g$ in $\tilde K_n$, 
\item $||(1-u_n)d_nh_2|| \leq 2^{-n}$,
\item $||[u_n, x]||\leq 2^{-n}$  for any $x$ in $\tilde\Delta$ 
\item $||[u_n, g(\phi)]||\leq 2^{-n}$ for any $\phi$ in $Y$ and for any $g$ in $K_n$,
\item $||(1-u_n)d_n[x, g(\phi)]||\leq 2^{-n}$ for any $x$ in $\tilde\Delta'$, $\phi$ in $Y$ and for any $g$ in $G$,
\item $||[u_n, g(c)]||\leq \frac{2^{-n}}{\mu_G(K_n)}$ for $g\in  K_n$ so that 
\[
||\left( \int_{g\in K_n}g(c)[u_n, g(c)]d\mu_G(g) \right)|| \leq 2^{-n}.
\]
\end{enumerate}
For the condition $(e)$, we can achieve this for any $g$ in $G$ because we have 
\[
d_n[x, g(\phi)] = [d_n, x]g(\phi) + [x, d_ng(\phi)]  \to 0
\]
as $g$ goes to infinity since $[d_n, x]$ is in $J$ and $d_n$ is in $A_1$ (c.f. Lemma \ref{lem_nondegen}).
Now, we set 
\[
M_1=\sum_{j\geq1}d_ju_jd_j, \,\, M_2=1-M_1=\sum_{j\geq1}d_j(1-u_j)d_j.
\]
Checking that the conditions (I), (II), (III) and (IV) hold for these $M_1, M_2$ goes as usual, so we leave it to the reader. Note that using (4) and (d), we have for any $g$ in $K_{n+1}- K_n$ and for any $\phi$ in $Y$,
\[
||[M_1, g(\phi)]||  = ||\sum_{j\geq1}d_j[u_j, g(\phi)]d_j|| = ||\sum_{j\geq n+1}d_j[u_j, g(\phi)]d_j|| \leq 2^{-n},
\]
where the series appearing here is absolutely convergent with summands in $J$. Condition (V) follows from this. Using (e), we have for any $x$ in $\tilde\Delta'$, $\phi$ in $Y$ and for any $g$ in $G$,
\[
||\sum_{j\geq n+1}d_j(1-u_j)d_j [x, g(\phi)]|| \leq 2^{-n},
\]
where the series appearing here is absolutely convergent with summands in $J$. For any $j\geq1$, $x$ in $\tilde\Delta'$ and for any $\phi$ in $Y$, we also have 
\[
d_j [x, g(\phi)] = [d_j, x]g(\phi) + [x, d_jg(\phi)] \to 0
\]
as $g$ in $G$ goes to infinity as mentioned before. Condition (VI) follows from these. Using (4), we have
\[
\int_{g \in G} g(c)M_1 g(c)d{\mu_G}(g) - M_1 = \sum_{j\geq1}d_j \left( \int_{g\in K_j}g(c)[u_j, g(c)]d\mu_G(g) \right)d_j.
\]
Using (f), we see that the series appearing here is absolutely convergent with summands in $J$. Condition (VII) follows from this.
\end{proof}
\begin{remark}\label{rem_below} In this proof, we took a compact subset  $Y$ of $C_0(E)$ consisting of functions $\phi$ such that $\{\phi\neq 0\}$ is contained in $\{c\neq 0\}$ and such that the translate $GY$ by $G$ generates $C_0(E)$. To see such $Y$ exists, we just need to know that any function $\phi$ in $C_0(E)$ can be written as 
\[
\int_{g\in G}g(c)^2\phi d\mu_G(g)=\int_{g\in G}g(c^2g^{-1}(\phi)) d\mu_G(g).
\]
In practice, we can take $Y$ to be just $\{c\}$ in the following sense. We can generate a $G$-subalgebra $C_0(E')$ of $C_0(E)$ by the translate $Gc$ by $G$ of the single function $c$. One can check that the spectrum $E'$ of this subalgebra is a proper $G$-space with a surjective $G$-equivariant map from $E$ to $E'$. Since $E$ is a co-compact model of the universal proper $G$-space, it follows that so is $E'$. Moreover, the function $c$ defines a cut-off function on $E'$. 
\end{remark}

\begin{theorem}\label{thm_implies}  If the gamma element $\gamma$ exists, the element $\gamma$ is represented by some $G$-equivariant Kasparov cycle with property $(\gamma)$.
\end{theorem}
\begin{proof} We prove that there is a representative $(H, T)$ of $\gamma$ which satisfies the conditions (ii.i) and (ii.ii) for property $(\gamma)$. Stabilizing $P$ by the algebra of compact operators if necessary, without loss of generality, we assume that $\gamma=\delta\otimes_Pd$ where $P$ is a proper, graded $G$-$\Calg$, that the dual Dirac element $\delta$ is represented by a cycle $(P,b)$ for $KK^G(\mathbb{C}, P)$ where $b$ is an odd, bounded, self-adjoint element in the multiplier algebra $M(P)$ of $P$ satisfying $1-b^2\in P$ and $g(b)-b\in P$ for any $g$ in $G$ and that the Dirac element $d$ is represented by a cycle $(H, F)$ for $KK^G(P, \mathbb{C})$ where $F$ is an odd, bounded, self-adjoint operator on a graded $G$-Hilbert space $H$ equipped with a non-degenerate representation of $P$ satisfying $a(1-F^2)\in K(H)$, $a(g(F)-F)\in K(H)$ and $[a, F] \in K(H)$ for any $g$ in $G$ and for any $a\in P$, where the commutator here is the graded commutator. We apply Theorem \ref{thm_tech} for the algebra $J=K(H)$, the proper algebra $P$, the algebra $A_1=C^\ast(P, J)$, the algebra $A_2$ generated by the elements $(1-F^2)$, $g(F)-F$ for $g$ in $G$, $[b, F]$ and $F'-F$ where
\[
F'=\int_{g\in G}g(c)Fg(c) d\mu_G(g),
\]
and $\Delta=\{b, F\}, \Delta'=\{F\}$. Let $M_1, M_2$ be operators on $H$ satisfying the conditions (I) to (VII) of Theorem \ref{thm_tech}. We define
\[
T=M_1^{\frac14}bM_1^{\frac14} + M_2^{\frac14} F M_2^{\frac14}.
\]
Checking that the pair $(H, T)$ is a Kasparov product of $(A,b)$ and $(H, F)$ goes as usual, so we leave it to the reader. We show that the operator $T$ satisfies the conditions (ii.i) and (ii.ii) for property $(\gamma)$ with respect to the representation of $C_0(E)$ on $H$ coming from that of $P$. Since $[M_1, b]$ and $[M_2, F]$ are in $K(H)$, by Lemma \ref{lem_property_alg}, it suffices to show that operators $M_i$, $M_1b$ and $M_2F$ in place of $T$ satisfy the conditions (ii.i) and (ii.ii). All of these follow immediately from conditions (V), (VI), (VII) of $M_i$ except the condition (ii.ii) for $M_2F$:
\[
\int_{g \in G}g(c)M_2Fg(c)d\mu_G(g) - M_2F \in K(H).
\]
Since $M_2(F'-F)$ and $\int_{g \in G}g(c)M_2g(c)d\mu_G(g) - M_2$ are in $K(H)$, we have
\[
\left(\int_{g \in G}g(c)M_2g(c)d\mu_G(g)\right)\left(\int_{g \in G}g(c)Fg(c)d\mu_G(g)\right) - M_2F \in K(H).
\]
The first term in this expression is equal to 
\[
\int_{(g, g') \in G\times G_0}g(c)M_2g(c)(gg')(c)F(gg')(c)d\mu_{G\times G},
\]
where $G_0$ is a compact subset of $G$ such that $cg(c)=0$ unless $g \in G_0$. By Lemma \ref{lem_nondegen_int} and (VI), this term is, modulo compact operators, equal to 
\[
\int_{(g, g') \in G\times G_0}g(c)M_2Fg(c)(gg')(c^2)d\mu_{G\times G} = \int_{g \in G}g(c)M_2Fg(c)d\mu_{G}. 
\] 
The condition (ii.ii) for $M_2F$ follows from these. We conclude that the cycle $(H, T)$ represents the gamma element $\gamma$ and satisfies the conditions (ii.i), (ii.ii) for property $(\gamma)$. Thus, the cycle $(H, T)$ is a representative of $\gamma$ which has property $(\gamma)$. 
\end{proof}

\section{The $(\gamma)$-element: uniqueness and $(\gamma)^2=(\gamma)$} 
We show that a Kasparov cycle with property $(\gamma)$ is unique up to homotopy. We define the $(\gamma)$-element in the Kasparov ring $R(G)$ and show that the $(\gamma)$-element is an idempotent. 
\begin{definition} \label{def_proper} We say that a $G$-equivariant Kasparov cycle $(H, T)$ is proper if the condition (ii) (hence, (ii.i) and (ii.ii)) of property $(\gamma)$ is satisfied.
\end{definition}

We next recall the definition of cycles for $G$-equivariant $K$-homology group $KK^G(B, \mathbb{C})$ for a $G$-$\Calg$ $B$.
\begin{definition}(\cite{Ka88}) For a separable $G$-$\Calg$ $B$, a $G$-equivariant Fredholm $B$-module is a triple $(\pi, H, T)$ where $H$ is a separable, graded $G$-Hilbert space equipped with a $G$-equivariant representation $\pi$ of $B$ and $T$ is an odd, bounded, self-adjoint operator on $H$ satisfying the following three conditions.
\[
\text{Fredholm condition:} \,\,\,\,\, b(1-T^2) \in K(H) \,\,\, \text{for any $b \in B$}.
\]
\[
\text{Almost-$G$-equivariance:} \,\,\,\, b(g(T)-T) \in K(H) \,\,\, \text{for any $g\in G$, $b \in B$}.
\]
\[
\text{Pseudo-locality:} \,\,\,\,  [b, T] \in K(H) \,\,\, \text{for any $b \in B$}.
\]
Here, we simply write $b$ for $\pi(b)$.
\end{definition}
The abelian group $KK^G(B, \mathbb{C})$ is defined as homotopy equivalence classes $[\pi, H, T]$ of Fredholm $B$-modules $(\pi, H, T)$. See \cite{Ka88}, \cite{Bl98} for more details and for more general groups $KK_\ast^G(B, A)$ defined by Fredholm $B$-$A$-modules.  

Let $(H, T)$ be a proper $G$-equivariant Kasparov cycle and denote by $\pi_0$, the representation of $C_0(E)$ on $H$ witnessing the condition (ii) for property $(\gamma)$. We also set $\pi_{G, H}$ to be the the representation of $G$ on $H$. Let $\bar H$ be the tensor product 
\[\bar H= H\otimes L^2(G).
\] with the $G$-Hilbert space structure given by the tensor product representation $1\otimes \rho_G$ of the trivial representation of $G$ on $H$ and the right regular representation of $G$ on $L^2(G)$. We equip $\bar H$ with the $G$-equivariant representation $\bar\pi_0\rtimes(\pi_{G, H}\otimes \lambda_G)$ of the reduced crossed product algebra $C_0(E)\rtimes_rG$ with the trivial $G$-action where $\bar\pi_0$ sends $\phi$ in $C_0(E)$ to $\pi_0(\phi)\otimes1$ on $\bar H$ and where $\lambda_G$ is the left-regular representation of $G$ on $L^2(G)$. Define an odd, bounded, self-adjoint operator $\bar T$ on $\bar H$ by
\[
\bar T = (g(T))_{g \in G}.
\]
\begin{proposition} \label{prop_barcycle} The triple $(\bar\pi_0\rtimes(\pi_{G, H}\otimes \lambda_G), \bar H, \bar T)$ is a $G$-equivariant Fredholm $C_0(E)\rtimes_rG$-module.
\end{proposition}
\begin{proof} We give a proof for a discrete group $G$. The general case is left as an exercise. First note that $1-T^2$ is compact and that the representation $\pi_0$ is non-degenerate on $H$. Thus, using Lemma \ref{lem_nondegen}, we have
\[
\text{the function} \,\,\, (g \mapsto \phi(1-g(T)^2)) \in C_0(G, K(H)) \,\,\, \text{for any $\phi \in C_0(E)$}.
\]
The Fredholm condition for the triple follows from this. Secondly, since $h(T)-T$ is compact for any $h$ in $G$, again by Lemma \ref{lem_nondegen}, we have 
\[
\text{the function} \,\,\, (g \mapsto \phi(gh(T)-g(T))) \in C_0(G, K(H)) \,\,\, \text{for any $\phi \in C_0(E)$}
\]
for any $h$ in $G$. The almost $G$-equivariance follows from this.  Since the cycle $(H, T)$ satisfies the condition (ii.i) of property $(\gamma)$, we have
\[
\text{the function}\,\,\, (g \mapsto [\phi, g(T)]) \in C_0(G, K(H)) \,\,\, \text{for any $\phi \in C_0(E)$}.
\]  
Also, the operator $\bar T$ commutes with the representation $\pi_{G, H}\otimes \lambda_G$ of $G$. These implies the pseudo-locality with respect to $C_0(E)\rtimes_rG$.
\end{proof}

Let $x$ in $R(G)$ be an element which is represented by a proper cycle $(H, T)$. We set $\bar x$ to be the element in $KK^G(C_0(E)\rtimes_r G, \mathbb{C})$ represented by the $G$-equivariant Fredholm $C_0(E)\rtimes_rG$-module $(\bar\pi_0\rtimes(\pi_{G, H}\otimes \lambda_G), \bar H, \bar T)$. For a cut-off function $c$ on $E$, the cut-off projection $p_c$ in $C_0(E)\rtimes_rG$ is defined as 
\[
p_c= \int_{g\in G}cg(c)u_gd\mu_G(g).
\]
This defines the element $[p_c]$ in $KK(\mathbb{C}, C_0(E)\rtimes_rG)$. We note that the element $[p_c]$ is defined independently of the choice of $c$. Note that the Kasparov product defines the pairing
\[
KK(\mathbb{C}, C_0(E)\rtimes_rG) \times KK^G(C_0(E)\rtimes_rG, \mathbb{C}) \to KK^G(\mathbb{C}, \mathbb{C})
\]
which in particular maps the pair $([p_c], \bar x)$ to the product $[p_c]\otimes_{C_0(E)\rtimes_rG} \bar x$ in $R(G)$.
\begin{proposition} \label{prop_barcycle_back} We have $[p_c]\otimes_{C_0(E)\rtimes_rG} \bar x = x$.
\end{proposition}
\begin{proof} Again, for simplicity, we only give a proof for a discrete group $G$. The product $[p_c]\otimes_{C_0(E)\rtimes_rG} \bar x$ is represented by the cycle $(p_c\bar H, p_c\bar Tp_c)$ where we simply write by $p_c$ the projection $\bar\pi_0\rtimes(\pi_{G, H}\otimes \lambda_G)(p_c)$ represented on $\bar H$. Without loss of generality, let us suppose that the cut-off function $c$ is the one witnessing the condition (ii.ii) for property $(\gamma)$ of the proper cycle $(H,T)$ representing $x$. We have an isomorphism $H \cong p_c\bar H$ of $G$-Hilbert spaces given by the map
\[
v \mapsto \sum_{h \in G}\pi_0(c)h(v)\otimes\delta_h \,\,\, \text{for $v$ in $H$},
\]
whose inverse is given by (the restriction of)
\[
(v_h)_{h \in G} \mapsto \sum_{h \in G}h^{-1}(\pi_0(c)v_h) \,\,\, \text{for $(v_h)_{h\in G}$ in $\bar H$}.
\]
Via this isomorphism, the cycle $(p_c\bar H, p_c\bar Tp_c)$ is isomorphic to the cycle $(H, T')$ where
\[
T'=\sum_{h\in G}h(c)Th(c)
\]
which is equal to $T$ modulo compact operators by the condition (ii.ii) of the cycle $(H, T)$. We see that $(H, T')$ is homotopic to $(H, T)$. The claim follows from this. 
\end{proof} 

Now, let $x$ in $R(G)$ be an element which is represented by a proper cycle $(H, T)$, $\bar x$ be the corresponding element in $KK^G(C_0(E)\rtimes_rG, \mathbb{C})$ as before, and $y$ be an arbitrary element in $R(G)$ which is represented by a cycle $(H_1, T_1)$. Recall that we have the amplification map 
\[
\sigma_B\colon KK_\ast^G(\mathbb{C}, \mathbb{C}) \to KK_\ast^G(B, B)
\]
for a $G$-$\Calg$ $B$ and that we have the descent map
\[
j^G_r\colon KK_\ast^G(A, B) \to KK_\ast(A\rtimes_rG, B\rtimes_rG)
\]
for $G$-$\Calgs$ $A$ and $B$. We refer \cite{Ka88} for details. The element $j^G_r(\sigma_{C_0(E)}(y))$ in $KK(C_0(E)\rtimes_rG, C_0(E)\rtimes_rG)$ acts on $KK^G(C_0(E)\rtimes_rG, \mathbb{C})$ from left by the Kasparov product
\[
K(C_0(E)\rtimes_rG, C_0(E)\rtimes_rG) \times KK^G(C_0(E)\rtimes_rG, \mathbb{C}) \to KK^G(C_0(E)\rtimes_rG, \mathbb{C}).
\]
\begin{proposition} \label{prop_proper_product} We have $\bar x \otimes_\mathbb{C}y= j^G_r(\sigma_{C_0(E)}(y))\otimes_{C_0(E)\rtimes_rG} \bar x$.
\end{proposition}
\begin{proof} For simplicity, we only give a proof for a discrete group $G$. The Kasparov product $\bar x \otimes_\mathbb{C}y$ is represented by the $G$-equivariant Fredholm $C_0(E)\rtimes_rG$-module
\[
((\pi_0\otimes1\otimes1)\rtimes(\pi_{G, H}\otimes1\otimes \lambda_G), H\hat\otimes H_1\otimes L^2(G) , F):
\] 
the Hilbert space $H\hat\otimes H_1\otimes L^2(G)$ is (isomorphic to) the graded tensor product of $\bar H$ and $H_1$ whose $G$-Hilbert space structure given by the tensor product representation $1\otimes \pi_{G, H_1}\otimes \rho_G$ ($\pi_{G, H_1}$ is the $G$-action on $H_1$), and 
\[
F = M^{\frac14}_1(g(T)\hat\otimes1)_{g\in G}M^{\frac14}_1 + M^{\frac14}_2(1\hat\otimes T_1)_{g\in G} M^{\frac14}_2
\]
where we take $M_1, M_2$ as degree $0$ positive operators on $H\hat\otimes H_1\otimes L^2(G)$ which we obtain by applying the usual Technical Theorem \cite[Theorem 1.4]{Ka88} so that $M_1+M_2=1$ and for $i=1,2$,
\begin{enumerate}
\item $g(M_i)-M_i \in K(H\hat\otimes H_1\otimes L^2(G))$,
\item $M_1x \in K(H\hat\otimes H_1\otimes L^2(G)) $ for any $x$ in $A_1$,
\item $M_2x \in K(H\hat\otimes H_1\otimes L^2(G)) $ for any $x$ in $A_2$,
\item $[M_i, x] \in K(H\hat\otimes H_1\otimes L^2(G)) $ for any $x$ in $\Delta$,
\end{enumerate}
for the separable $G$-$\Calg$ $A_1$ generated by $C_0(G, K(H))$ and the product $C_0(G, K(H))\cdot C_0(E)\rtimes_rG$, the separable $G$-$\Calg$ $A_2$ generated by the operators $(1\hat\otimes(1-T^2_1))_{g \in G}$, $(1\hat\otimes(g^{-1}(T_1)-T_1))_{g\in G}$ and $(1\hat\otimes(h(T_1)-T_1))_{g\in G}$ for $h$ in $G$, and for the separable subset 
\[
\Delta=C_0(E)\rtimes_rG \cup \{(g(T)\hat\otimes1)_{g\in G}, (1\hat\otimes T_1)_{g \in G}\}
\] which derives $A_1$. Notice, we included the extra condition 
\begin{enumerate}[(iii)']
\item  $M_2 (1\hat\otimes(g^{-1}(T_1)-T_1))_{g\in G} \in K(H\hat\otimes H_1\otimes L^2(G))$
\end{enumerate}
for $M_2$ which is not necessary if we just want to produce the Kasparov product. Thanks to this extra condition (iii)', we see that the operator $F$ is equal to
\[
F' = M^{\frac14}_1(g(T)\hat\otimes1)_{g\in G}M^{\frac14}_1 + M^{\frac14}_2(1\hat\otimes g^{-1}(T_1))_{g\in G} M^{\frac14}_2
\]
modulo compact operators. Thus, we see that the product $\bar x \otimes_\mathbb{C}y$ is represented by the $G$-equivariant Fredholm $C_0(E)\rtimes_rG$-module
\begin{equation}\label{triple1}
((\pi_0\otimes1\otimes1)\rtimes(\pi_{G, H}\otimes1\otimes \lambda_G), H\hat\otimes H_1\otimes L^2(G) , F').
\end{equation}
Now, we use the following isomorphism of the Hilbert space $H\hat\otimes H_1\otimes L^2(G)$
\[
U\colon v\otimes v_1\otimes \delta_g \to   v\otimes gv_1\otimes \delta_g \,\,\, \text{for $v\otimes v_1\otimes \delta_g$ in $H\hat\otimes H_1\otimes L^2(G)$}.
\]
By the conjugation by $U$, the triple \eqref{triple1} is isomorphic to 
\begin{equation}\label{triple2}
((\pi_0\otimes1\otimes1)\rtimes(\pi_{G, H}\otimes\pi_{G, H_1}\otimes \lambda_G), H\hat\otimes H_1\otimes L^2(G) , UF'U^\ast)
\end{equation}
where now the $G$-Hilbert space structure on $H\hat\otimes H_1\otimes L^2(G)$ is given by the tensor product representation $1\otimes1\otimes \rho_G$ of the trivial action on $H\hat\otimes H_1$ and the right-regular representation of $G$ on $L^2(G)$. We have
\[
UF'U^\ast = M^{\frac14}_3(g(T)\hat\otimes1)_{g\in G}M^{\frac14}_3 + M^{\frac14}_4(1\hat\otimes T_1)_{g\in G} M^{\frac14}_4.
\] 
where $M_i=UM_iU^\ast$. On the other hand, it is now not so hard to see that the product $j^G_r(\sigma_{C_0(E)}(y))\otimes_{C_0(E)\rtimes_rG} \bar x$ is represented by the triple \eqref{triple2}. 
\end{proof}

\begin{corollary}\label{cor_proper_product} Let $x$ be an element in $R(G)$ which is represented by a proper cycle. Let $y$ be an element in $R(G)$ such that $y=1_K$ in $R(K)$ for any compact subgroup $K$ of $G$. Then, we have \[
y\otimes_{\mathbb{C}}x = x\otimes_{\mathbb{C}}y= x.
\]
\end{corollary}
\begin{proof} The Kasparov product is graded commutative and associative. Thus, combined with Proposition \ref{prop_barcycle_back}, we have
\[
y\otimes_{\mathbb{C}}x = x\otimes_{\mathbb{C}}y= ([p_c]\otimes_ {C_0(E)\rtimes_rG} \bar x)\otimes_{\mathbb{C}}y = [p_c] \otimes_ {C_0(E)\rtimes_rG}(\bar x \otimes_{\mathbb{C}}y).
\]
By Proposition \ref{prop_proper_product}, we have
\[
\bar x \otimes_{\mathbb{C}}y =  j^G_r(\sigma_{C_0(E)}(y))\otimes_{C_0(E)\rtimes_rG} \bar x = \bar x
\]
since the condition $y=1_K$ in $R(K)$ for any compact subgroup $K$ of $G$ implies that $\sigma_{C_0(E)}(y)=\mathrm{id}_{C_0(E)}$ and hence $j^G_r(\sigma_{C_0(E)}(y))=\mathrm{id}_{C_0(E)\rtimes_rG}$ (see \cite[Corollary 7.2]{MN06} for this fact). Thus, we have
\[
y\otimes_{\mathbb{C}}x = x\otimes_{\mathbb{C}}y= [p_c] \otimes_ {C_0(E)\rtimes_rG}(\bar x \otimes_{\mathbb{C}}y) =  [p_c] \otimes_ {C_0(E)\rtimes_rG}\bar x = x. 
\]
\end{proof}

\begin{theorem}\label{thm_unique} Let $x_1$ and $x_2$ be elements in $R(G)$ which are represented by cycles with property $(\gamma)$. Then, we have
\[
x_1=x_2.
\]
That is, a $G$-equivariant Kasparov cycle with property $(\gamma)$ is unique up to homotopy.
\end{theorem}
\begin{proof} Using Corollary \ref{cor_proper_product} for $x=x_1, y=x_2$ and for $x=x_2, y=x_1$, we have
\[
x_1= x_1\otimes_{\mathbb{C}}x_2= x_2.
\] \end{proof}
Thanks to Theorem \ref{thm_unique}, the following notion is well-defined.
\begin{definition} \label{def_gamma_element} Suppose there is a $G$-equivariant Kasparov cycle $(H, T)$ with property $(\gamma)$. We define the $(\gamma)$-element for $G$ to be the unique element in $R(G)$ which is represented by a $G$-equivariant  Kasparov cycle $(H, T)$ with property $(\gamma)$.
\end{definition}

Thus, if the $\gamma$-element exists for a group $G$, the two notions of $\gamma$-element and $(\gamma)$-element coincide, i.e. $\gamma$ is the $(\gamma)$-element.

\begin{theorem}\label{thm_idempotent}  The $(\gamma)$-element is an idempotent in the Kasparov ring $R(G)$.
\end{theorem}
\begin{proof} Using Corollary \ref{cor_proper_product} for $x=y=x_1$, we have
\[
x_1\otimes_{\mathbb{C}}x_1= x_1.
\]
\end{proof}

\section{$(\gamma)$-morphism} For a $G$-equivariant Kasparov cycle $(H, T)$ with property $(\gamma)$, we shall construct a natural map of Kasparov's $G$-equivariant $KK$-theory 
\[
\nu^{G, T}_A\colon KK_\ast(\mathbb{C}, A\rtimes_rG) \to KK^G_\ast(C_0(E), A)
\]
for any separable $G$-$\Calg$ $A$. Let $(H, T)$ be a $G$-equivariant Kasparov cycle with property $(\gamma)$ and denote by $\pi_0$, the representation of $C_0(E)$ on $H$ witnessing the condition (ii) for property $(\gamma)$. We also set $\pi_{G, H}$ to be the representation of $G$ on $H$. Let $\tilde H$ be the tensor product
\[
\tilde H = H\otimes L^2(G)
\]
with the $G$-Hilbert space structure given by the tensor product representation $\pi_{G, H}\otimes\lambda_G$ of the representation $\pi_{G, H}$ and the left-regular representation. We equip $\tilde H$ with the $G$-equivariant  representation $\pi_0\otimes \rho^G$ of $C_0(E)\otimes C^\ast_r(G)$ where $\rho^G$ is the right regular representation of $C^\ast_r(G)$ on $L^2(G)$. Define an odd, self-adjoint $G$-equivariant operator $\tilde T$ on $\tilde H$ by
\[
\tilde T = (g(T))_{g\in G}.
\]
The proof of the following is analogous to the one for Proposition \ref{prop_barcycle}, so we leave it to the reader.
\begin{proposition} \label{prop_tildecycle} The triple $(\pi_0\otimes\rho^G, \tilde H, \tilde T)$ is a $G$-equivariant Fredholm $C_0(E)\otimes C^\ast_r(G)$-module.
\end{proposition}

\begin{remark} We only need the condition (ii.i) of property $(\gamma)$ for a cycle $(H, T)$ to define the Fredholm $C_0(E)\otimes C^\ast_r(G)$-module $(\pi_0\otimes\rho^G, \tilde H, \tilde T)$.
\end{remark}

Let us denote by $[\tilde T]$ the class in $KK^G(C_0(E)\otimes C^\ast_r(G), \mathbb{C})$ represented by the triple $(\pi_0\otimes\rho^G, \tilde H, \tilde T)$.
For any separable $G$-$\Calg$ $A$, in above construction, if we replace $L^2(G)$ by the $G$-Hilbert right $A$-module $L^2(G,A)$ and $\rho^G$ by the right regular representation $\rho^G_A$:
\[
a \mapsto (g(a))_{g\in G}, \,\,\, h \mapsto (\rho_h\colon (a_g)_{g\in G}\mapsto (a_{gh})_{g\in G}) \,\,\, \text{for $a\in A$, $h \in G$}
\]
of the reduced crossed product $A\rtimes_rG$ with trivial $G$-action, we can define a $G$-equivariant Fredholm $C_0(E)\otimes (A\rtimes_rG)$-$A$-module which defines the class $[\tilde T_A]$ in $KK^G(C_0(E)\otimes A\rtimes_rG, A)$.
\begin{definition}\label{def_gamma'_mor} Let $(H, T)$ be a Kasparov cycle with property $(\gamma)$. For any separable $G$-$\Calg$ $A$, we define a group homomorphism
\[
\nu^{G, T}_A\colon KK_\ast(\mathbb{C}, A\rtimes_rG) \to KK^G_\ast(C_0(E), A)
\]
as the one induced by the class $[\tilde T_A]$ in $KK^G(C_0(E)\otimes A\rtimes_rG, A)$ via the index pairing:
\[
KK_\ast(\mathbb{C}, A\rtimes_rG) \times KK^G(C_0(E)\otimes A\rtimes_rG, A) \to KK^G_\ast(C_0(E), A).
\]
We call this map $\nu^{G, T}_A$, the $(\gamma)$-morphism defined by $(H, T)$.
\end{definition}

\begin{remark} Again, the definition of the $(\gamma)$-morphism requires only the condition (ii.i) of property $(\gamma)$. Also, the definition implicitly depends not only on $(H, T)$ but also on the representation $\pi_0$ witnessing property $(\gamma)$.
\end{remark}

The $(\gamma)$-morphism is natural in the following sense.
\begin{proposition} \label{prop_natural} Let $(H, T)$ be a $G$-equivariant Kasparov cycle with property $(\gamma)$. For any element $\theta$ in $KK^G(A, B)$, the following diagram is commutative:
\begin{align*}
\xymatrix{
\nu^{G, T}_{A}\colon KK_\ast(\mathbb{C}, A\rtimes_rG) \ar[d]^{\theta\rtimes_r1_\ast} \ar[r] & KK^G_\ast(C_0(E), A) \ar[d]^{\theta_\ast}  \\
\nu^{G, T}_{B}\colon KK_\ast(\mathbb{C}, B\rtimes_rG) \ar[r] & KK^G_\ast(C_0(E), B)
}
\end{align*} 
where the horizontal maps are the $(\gamma)$-morphisms and the vertical ones are the ones induced by $\theta$.
\end{proposition}
\begin{proof} We can directly check this for the case when $\theta$ is a $G$-equivariant $\ast$-homomorphism from $A$ to $B$. The general case follows since any element in $KK^G(A, B)$ is the composition of $\ast$-homomorphisms and the inverse of $\ast$-homomorphisms in the category $KK^G$ (See Theorem 6.5 of \cite{Mey00}).
\end{proof}
\section{The Main Results}
As a simple application of our construction of the $(\gamma)$-morphism, we get some results on the Baum--Connes conjecture. For a separable $G$-$\Calg$ $A$, there is the so-called Baum-Connes assembly map:
\[
\mu^G_A\colon KK^G_\ast(C_0(E), A) \to KK_\ast(\mathbb{C}, A\rtimes_rG).
\]
\begin{conjecture} (Baum--Connes conjecture with coefficients for $G$, Baum, Connes and Higson \cite{BCH94}) The map $\mu^G_A$ is an isomorphism for any $A$.
\end{conjecture}
For introductions and details of the conjecture, we refer to \cite{BCH94}, \cite{Va02} and \cite{Ec17} to name a few. The weaker assertion that the map $\mu^G_A$ is injective for any $A$ is called the strong Novikov conjecture.

Now, take any Kasparov cycle $(H, T)$ with property $(\gamma)$. We first compute the composition 
\[
\mu^G_A\circ \nu^{G, T}_A\colon KK_\ast(\mathbb{C}, A\rtimes_rG) \to KK^G_\ast(C_0(E), A) \to KK_\ast(\mathbb{C}, A\rtimes_rG)
\]
of the assembly map and the $(\gamma)$-morphism. We recall from \cite{Ka88} that any Kasparov cycle $(H, T)$ defines an endomorphism on K-theory group $KK_\ast(\mathbb{C}, A\rtimes_rG)$ via the following composition of ring homomorphisms
\begin{equation}\label{eq_comp2}
KK_\ast^G(\mathbb{C}, \mathbb{C}) \to KK_\ast^G(A, A) \to KK_\ast(A\rtimes_rG, A\rtimes_rG)
\end{equation}
where the first map is the amplification map $\sigma_A$, the second map is the descent map $j^G_r$. Let us write $j^G_r(\sigma_A([H,T]))$, the image of $[H, T]$ by the map \eqref{eq_comp2}.

\begin{proposition} \label{prop_surj} Let $(H, T)$ be a Kasparov cycle with property $(\gamma)$. For any separable $G$-$\Calg$ $A$, the composition $\mu^G_A\circ \nu^{G, T}_A$ coincides with the action of $(H, T)$ on $KK_\ast(\mathbb{C}, A\rtimes_rG)$ defined via \eqref{eq_comp2}.
\end{proposition}
\begin{proof} For the simplicity, we only give a proof for the case where $G$ is discrete, $\ast=0$ and $A=\mathbb{C}$, since it is straightforward to generalize its argument. We compute the image of (the class) of a projection $p$ in the matrix algebra $M_n(C^\ast_r(G))$ by the composition $\mu^G_{\mathbb{C}}\circ \nu^{G, T}_{\mathbb{C}}$. We write $\tilde H$, $\tilde T$ and $\rho^G$ as in Section 4. Let $\pi_0$ be the representation of $C_0(E)$ on $H$ witnessing the condition (ii) of property ($\gamma$) for the cycle $(H, T)$. First of all, $\nu^{G, T}_{\mathbb{C}}(p)$ in $KK^G(C_0(E), \mathbb{C})$ is represented by the following Fredholm $C_0(E)$-module ($\rho^G$ is extended to $l^2(G)\otimes \mathbb{C}^n$):
\[
(\pi_0\cdot\rho^G(p), \tilde H\otimes \mathbb{C}^n, \tilde T)
\]
where $\phi$ in $C_0(E)$ is represented by $\pi_0(\phi)\rho^G(p)$. The descent map $j^G_r$ sends this module to a Fredholm $C_0(E)\rtimes_rG$-$C^\ast_r(G)$-module
\[
((\pi_0\rtimes_r1)\cdot\rho^G(p)\rtimes1, \tilde H\otimes \mathbb{C}^n\rtimes_rG, \tilde T\rtimes_r1)
\]
where we recall that the Hilbert right $C^\ast_r(G)$-module $\tilde H\otimes \mathbb{C}^n\rtimes_rG$ is spanned by vectors of the form
\[
v\rtimes_ru_g \,\,\,\,\,  \text{for $v \in \tilde H\otimes \mathbb{C}^n$ and $g\in G$},
\]
for an operator $T$ on $\tilde H\otimes \mathbb{C}^n$, the adjointable map $T\rtimes_r1$ is defined by
\[
T\rtimes_r1\colon v\rtimes_ru_g \mapsto Tv\rtimes_ru_g
\]
and that $\pi_0\rtimes_r1$ is the representation of $C_0(E)\rtimes_rG$ which sends $\phi$ in $C_0(E)$ to $\pi_0(\phi)\rtimes_r1$ and ``sends'' $h$ in $G$ to the left multiplication $u_h$,
\[
u_h\colon v\rtimes_ru_g \mapsto h(v)\rtimes_ru_{hg}.
\]

A cut-off projection $p_c$ in $C_0(E)\rtimes_rG$ is defined as a finite-sum
\[
p_c= \sum_{g\in G}cg(c)u_g.
\]
Here, we take $c$ to be a cut-off function on $E$ witnessing the condition (ii.ii) of property ($\gamma$).
We see that the assembly map sends $\nu^{G, T}_{\mathbb{C}}(p)$ to a Fredholm $\mathbb{C}$-$C^\ast_r(G)$-module
\[
((\pi_0\rtimes_r1)(p_c)\cdot\rho^G(p)\rtimes_r1, \tilde H\otimes \mathbb{C}^n\rtimes_rG, \tilde T\rtimes_r1)
\]
where the unit in $\mathbb{C}$ is represented by the projection $(\pi_0\rtimes_r1)(p_c)\cdot\rho^G(p)\rtimes_r1$. We have an isomorphism
\begin{equation}\label{eq_isom}
H\otimes \mathbb{C}^n\rtimes_rG \cong (\pi_0\rtimes_r1)(p_c) \tilde H\otimes \mathbb{C}^n\rtimes_rG
\end{equation}
of $C^\ast_r(G)$-modules given as follows. We denote by $\delta_h$, the delta-function on $G$. The isomorphism \eqref{eq_isom} is given by the map
\[
\xi\otimes v\rtimes_ru_g \mapsto \sum_{h\in G}\pi_0(c)h(\xi)\otimes\delta_h\otimes v \rtimes_ru_{hg} \,\,\,\,\, \text{($\xi \in H$, $v \in \mathbb{C}^n$, $g \in G$)}
\]
whose inverse is given by (the restriction of)
\[
(\xi_h)_{h \in G}\otimes v \rtimes_ru_g \mapsto \sum_{h \in G}h^{-1}(\pi_0(c)\xi_h)\otimes v\rtimes_r u_{h^{-1}g} \,\,\, \text{($(\xi_h)_{h\in G} \in \tilde H$, $v \in \mathbb{C}^n$, $g \in G$)}.
\]
Under this isomorphism \eqref{eq_isom}, the restriction $(\pi_0\rtimes_r1)(p_c)\tilde T\rtimes_r1(\pi_0\rtimes_r1)(p_c)$ of $\tilde T\rtimes_r1$ on $(\pi_0\rtimes_r1)(p_c)\tilde H\otimes \mathbb{C}^n\rtimes_rG$ is identified as $T'\rtimes_r1$ on $H\otimes \mathbb{C}^n\rtimes_rG$ where we set
\[
T' = \sum_{g \in G} g(c)T g(c).
\] 
Moreover, the restriction $(\pi_0\rtimes_r1)(p_c)\rho^G(p)\rtimes_r1(\pi_0\rtimes_r1)(p_c)$ of $\rho^G(p)\rtimes_r1$ is identified as the left multiplication of $p$ on $H\otimes \mathbb{C}^n\rtimes_rG$.

We conclude that the composition $\mu^G_{\mathbb{C}}\circ \nu^{G, T}_{\mathbb{C}}$ sends the projection $p$ to the Fredholm $\mathbb{C}$-$C^\ast_r(G)$-module $(p, H\otimes\mathbb{C}^n\rtimes_rG, T'\rtimes_r1)$ where the unit of $\mathbb{C}$ acts as $p$ by the left multiplication. This is nothing but the image of $p$ by the action of $(H, T')$. By the condition (ii.ii) of property $(\gamma)$, this action coincides with that of $(H, T)$.
\end{proof}

\begin{proposition} \label{prop_inj} Let $(H, T)$ be a Kasparov cycle with property $(\gamma)$. For any separable $G$-$\Calg$ $A$, the composition $ \nu^{G, T}_A\circ \mu^G_A$ is the identity on $KK^G_\ast(C_0(E), A)$.
\end{proposition}
\begin{proof} We first show that, for any $G$-$C^\ast$-algebra $A$ for which the assembly map $\mu^G_A$ is an isomorphism, the composition $ \nu^{G, T}_A\circ \mu^G_A$ coincides with the action of $(H, T)$ on $KK^G_\ast(C_0(E), A)$ defined by the Kasparov product
\[
KK^G(\mathbb{C}, \mathbb{C}) \times KK^G_\ast(C_0(E), A) \to KK^G_\ast(C_0(E), A).
\]
We consider the following composition
\begin{align} \label{diag_comp}
\xymatrix{
\mu^G_A \circ \nu^{G, T}_A \circ \mu^G_A\colon  KK^G_\ast(C_0(E), A) \ar[r] &  KK_\ast(\mathbb{C}, A\rtimes_rG)
}
\end{align} 
We know that that the composition $\mu^G_A\circ \nu^{G, T}_A$ coincides with $j^G_r(\sigma_A([H,T]))_\ast$ on $KK_\ast(\mathbb{C}, A\rtimes_rG)$ via \eqref{eq_comp2}. From this, we can deduce that the composition $\nu^{G, T}_A\circ \mu^G_A$ coincides with the action $[H, T]_\ast$ of $(H, T)$ on $KK^G_\ast(C_0(E), A)$. To see this, we write the composition \eqref{diag_comp} in two ways to have the following commutative diagram:
\begin{align} \label{diag_comp1}
\xymatrix{
KK^G_\ast(C_0(E), A) \ar[d]_{\cong}^{\mu^G_A} \ar[r]_{\nu^{G, T}_A\circ \mu^G_A} &   KK^G_\ast(C_0(E), A) \ar[d]_{\cong}^{\mu^G_A}\\
KK_\ast(\mathbb{C}, A\rtimes_rG) \ar[r]_{\mu^G_A\circ \nu^{G, T}_A}  & KK_\ast(\mathbb{C}, A\rtimes_rG).
}
\end{align}
The composition $\mu^G_A\circ \nu^{G, T}_A$ in the bottom is $j^G_r(\sigma_A([H,T]))_\ast$. Note that we have also the following commutative diagram
\begin{align} \label{diag_comp3}
\xymatrix{
KK^G_\ast(C_0(E), A) \ar[d]_{\cong}^{\mu^G_A} \ar[r]_{[H,T]_\ast} &   KK^G_\ast(C_0(E), A) \ar[d]_{\cong}^{\mu^G_A}\\
KK_\ast(\mathbb{C}, A\rtimes_rG) \ar[r]_{j^G_r(\sigma_A([H,T]))_\ast}  & KK_\ast(\mathbb{C}, A\rtimes_rG).
}
\end{align}
Since the vertical arrows $\mu^G_A$ are isomorphisms, comparing the two diagram \eqref{diag_comp1}, \eqref{diag_comp3}, we see that $\nu^{G, T}_A\circ \mu^G_A=[H, T]_\ast$. We next show $\nu^{G, T}_A\circ \mu^G_A=[H, T]_\ast$ for any $A$. For this, we recall the following very useful fact established by Meyer and Nest (see Theorem 5.2 of \cite{MN06}). For any $G$-$\Calg$ $A$, there are a ``proper'' algebra $\tilde A$ for which the assembly map $\mu^G_{\tilde A}$ is an  isomorphism and a morphism $f$ from $\tilde A$ to $A$ in the category $KK^G$ which induces an isomorphism on the left-hand side of the Baum--Connes conjecture. That is, we have the following diagram:
\begin{align} \label{diag_5}
\xymatrix{
 KK_\ast^G(C_0(E), \tilde A) \ar[d]^{\mu^G_{\tilde A}}_{\cong} \ar[r]^{\cong}_{f_\ast} & KK_\ast^G(C_0(E), A) \ar[d]^{\mu^G_A} \\
 KK_\ast(\mathbb{C}, \tilde A\rtimes_rG)  \ar[r]_{f\rtimes_r1_\ast} & KK_\ast(\mathbb{C}, A\rtimes_rG)  
}
\end{align} We compose the diagram \eqref{diag_5} with the $(\gamma)$-morphisms to get the following diagram
\begin{align} \label{diag_6}
\xymatrix{
 KK_\ast^G(C_0(E), \tilde A) \ar[d]^{\nu^{G, T}_{\tilde A}\circ\mu^G_{\tilde A}} \ar[r]^{\cong}_{f_\ast} & KK_\ast^G(C_0(E), A) \ar[d]^{\nu^{G, T}_{A}\circ\mu^G_{A}}  \\
 KK_\ast^G(C_0(E), \tilde A) \ar[r]^{\cong}_{f_\ast} & KK_\ast^G(C_0(E), A).   \\
}
\end{align} This diagram is commutative since the $(\gamma)$-morphisms are natural (Lemma \ref{prop_natural}). The previous argument shows that $\nu^{G, T}_{\tilde A}\circ\mu^G_{\tilde A}=[H, T]_\ast$. We can compare the diagram \eqref{diag_6} to the diagram 
\begin{align*} 
\xymatrix{
 KK_\ast^G(C_0(E), \tilde A) \ar[d]^{[H,T]_\ast} \ar[r]^{\cong}_{f_\ast} & KK_\ast^G(C_0(E), A) \ar[d]^{[H,T]_\ast}  \\
 KK_\ast^G(C_0(E), \tilde A) \ar[r]^{\cong}_{f_\ast} & KK_\ast^G(C_0(E), A).   \\
}
\end{align*} 
Since the horizontal arrows $f_\ast$ are isomorphisms, we now deduce that the composition $\nu^{G, T}_{\tilde A}\circ\mu^G_{\tilde A}$ coincides with the action $[H, T]_\ast$ for any $A$. It is a standard fact that the $G$-equivariant Kasparov cycle $(H, T)$ acts as the identity on the left-hand side $KK_\ast^G(C_0(E), A)$ of the Baum--Connes conjecture for any $A$, provided that the cycle $(H,T)$ is $K$-equivariantly homotopic to $1_K$ for any compact subgroup $K$ of $G$ (see \cite{MN06}). We now see that the composition $\nu^{G, T}_{\tilde A}\circ\mu^G_{\tilde A}$ is the identity for any $A$.
\end{proof}

Now, we obtain our main results.
\begin{theorem} \label{thm_BCCmain} Suppose there is a $G$-equivariant Kasparov cycle $(H, T)$ with property $(\gamma)$. Then:
\begin{enumerate}[(i)]
\item the strong Novikov conjecture holds for $G$, i.e. the assembly map $\mu^G_A$ is split-injective for any $A$,
\item the $(\gamma)$-morphism $\nu^{G, T}_A$ is a left-inverse of the assembly map $\mu^G_A$,
\item the assembly map $\mu^G_A$ is an isomorphism if and only if the cycle $(H, T)$ acts as the identity on the right-hand side group $KK_\ast(\mathbb{C}, A\rtimes_rG)$  via the composition \eqref{eq_comp2}.
\end{enumerate}
\end{theorem}
\begin{proof} The first and the second claims are proved in Proposition \ref{prop_inj}. It follows that the assembly map $\mu^G_A$ is an isomorphism if and only if the composition $\mu^G_A\circ\nu^{G, T}_A$ is the identity. By Proposition \ref{prop_surj}, this is the case if and only if the cycle $(H, T)$ acts as the identity on the right-hand side group $KK_\ast(\mathbb{C}, A\rtimes_rG)$.
\end{proof}

\begin{corollary} \label{cor_BCCmain}  Suppose there is a $G$-equivariant Kasparov cycle $(H, T)$ with property $(\gamma)$ which acts surjectively on the right-hand side group $KK_\ast(\mathbb{C}, A\rtimes_rG)$ via the composition \eqref{eq_comp2} for any $A$. Then, the Baum--Connes conjecture with coefficients holds for $G$. The $(\gamma)$-morphism $\nu^{G, T}_A$ is the inverse of the assembly map $\mu^G_A$ for any $A$.
\end{corollary} 

\begin{corollary} \label{cor_BCCmain2} Suppose there is a $G$-equivariant Kasparov cycle $(H, T)$ with property $(\gamma)$ which is homotopic to $1_G$. Then, the Baum--Connes conjecture with coefficients holds for $G$. The $(\gamma)$-morphism $\nu^{G, T}_A$ is the inverse of the assembly map $\mu^G_A$ for any $A$.
\end{corollary}

\begin{remark}
In order to prove the third claim in Theorem \ref{thm_BCCmain}, we could have skipped Proposition \ref{prop_inj} and showed  that the assembly map $\mu^G_A$ is an isomorphism for any $A$ if $(H, T)$ acts as a surjective endomorphism on the right-hand side group $KK_\ast(\mathbb{C}, A\rtimes_rG)$ for any $A$ by the following ``surjectivity implies injectivity'' principle. Indeed, if $(H, T)$ acts as a surjective endomorphism on the right-hand side group $KK_\ast(\mathbb{C}, A\rtimes_rG)$ for any $A$, by Proposition \ref{prop_surj}, we know that $\mu^G_A$ is surjective for any $A$. Now, we use again the result by Meyer and Nest. Let $\tilde A$ and $f$ be as in the proof of Proposition \ref{prop_inj}. Furthermore, we can take the mapping cone $C_f$ of $f$ and obtain the following diagram of six-term exact sequences and the assembly maps
\begin{align*} 
\xymatrix{
\ar[r] & KK_\ast^G(C_0(E), C_f) \ar[d]^{\mu^G_{C_f}}  \ar[r] & KK_\ast^G(C_0(E), \tilde A) \ar[d]^{\mu^G_{\tilde A}}_{\cong} \ar[r]^{\cong}_{f_\ast} & KK_\ast^G(C_0(E), A) \ar[d]^{\mu^G_A} \ar[r] &  \\
\ar[r] & KK_\ast(\mathbb{C}, C_f\rtimes_rG)  \ar[r] & KK_\ast(\mathbb{C}, \tilde A\rtimes_rG)  \ar[r]_{f\rtimes_r1_\ast} & KK_\ast(\mathbb{C}, A\rtimes_rG)  \ar[r] &
}
\end{align*} where the group $KK_\ast^G(C_0(E), C_f)$ is zero since $f_\ast$ is an isomorphism. Thus, if we know the surjectivity of the assembly map for all coefficients, in particular for $C_f$, the ``obstruction'' $K_\ast(C_f\rtimes_rG)$ vanishes. It follows that $\mu^G_A$ is an isomorphism for any $A$. We remark that for this argument, we do not need the condition (i) of property $(\gamma)$. Indeed, the definition of the $(\gamma)$-morphism and Proposition \ref{prop_surj} work for any proper $G$-equivariant Kasparov cycle $(H, T)$. Hence, we have the following:
\begin{theorem}  Suppose that for any separable $G$-$\Calg$ $A$, there is a proper $G$-equivariant Kasparov cycle $(H, T)$ which acts surjectively on the right-hand side group $KK_\ast(\mathbb{C}, A\rtimes_rG)$ via the composition \eqref{eq_comp2}. Then, the Baum--Connes conjecture with coefficients holds for $G$.
\end{theorem}
\end{remark}

\section{Unbounded Cycles and Property $(\gamma)$}
Let $D$ be an odd, unbounded, (essentially) self-adjoint operator on $H$ with compact resolvent which is almost $G$-equivariant: i.e. $G$ preserves the domain of $D$ and $g(D)-D$ is bounded for all $g$ in $G$. It is a fact that the pair $(H, T)$ is a $G$-equivariant Kasparov cycle, where $T$ is the bounded transform of $D$:
\[
T=\frac{D}{(1+D^2)^{\frac12}}.
\]

\begin{theorem} \label{thm_gamma} Assume that there is a non-degenerate representation of the $G$-$\Calg$ $C_0(E)$ on $H$ and that there is a dense $G$-subalgebra $B$ of $C_c(E)$ (compactly supported functions) which preserves the domain of $D$, such that for any $b$ in $B$, there is $\chi$ in $C_c(E)$ so that for any $g$ in $G$,
\[
[D, g(b)]=b_gg(\chi)
\]
where $b_g$ are operators on $H$ which are uniformly bounded in $g$ in $G$. Assume further that there is a cut-off function on $E$ in $B$. Then, the Kasparov cycle $(H, T)$ is proper, i.e. satisfies the conditions (ii.i) and (ii.ii) for property $(\gamma)$. Hence, if we further assume that the cycle $(H, T)$ is $K$-equivariantly homotopic to $1_K$ for any compact subgroup $K$ of $G$, the cycle $(H, T)$ has property $(\gamma)$.
\end{theorem}
\begin{proof} For any $b$ in $B$, let $\chi$ in $C_c(E)$ and $b_g$ as given. Using the following formula (in the strong operator topology)
\[
T=\frac{2}{\pi}\int^{\infty}_0\left(\frac{D}{1+\lambda^2+D^2} \right)d\lambda = \pi^{-1}\int^{\infty}_0\left(\frac{1}{D+\sqrt{1+\lambda^2}i} + \frac{1}{D-\sqrt{1+\lambda^2}i}  \right)d\lambda
\] 
and the formula
\[
[g(b), (D\pm\sqrt{1+\lambda^2}i)^{-1}]= (D\pm\sqrt{1+\lambda^2}i)^{-1} [D, g(b)] (D\pm\sqrt{1+\lambda^2}i)^{-1},
\]
we see that in order to show that for any $b$ in $B$,
\[
\text{the function}\,\,\, (g \mapsto [g(b), T])\,\, \text{belongs to}\,\, C_0(G, K(H)),
\]
it suffices to show the following claim that operators
\[
\int^{\infty}_0\left((D\pm\sqrt{1+\lambda^2}i)^{-1} (b_gg(\chi)) (D\pm\sqrt{1+\lambda^2}i)^{-1}\right)d\lambda
\]
are compact operators whose norms vanish as $g$ goes to infinity. Using 
\[
\left( g(\chi)(D\pm\sqrt{1+\lambda^2}i)^{-\frac{1}{2}}\right)\left(g(\chi))(D\pm\sqrt{1+\lambda^2}i)^{-\frac{1}{2}}\right)^\ast \leq g(\chi)(D^2+1)^{-1} g(\chi)^\ast
\]
and Lemma \ref{lem_nondegen}, we can see that
\[
\sup_{\lambda \in [0,\infty]} ||(b_gg(\chi))(D\pm\sqrt{1+\lambda^2}i)^{-\frac{1}{2}}||
\]
vanishes as $g$ goes to infinity. With this in mind, we see that the claim holds since
\[
\int^{\infty}_0(1+\lambda^2)^{-\frac{3}{2}}d\lambda
\]
is absolutely convergent. Now, suppose $c$ is a cut-off function in $B$ and let $\chi$ in $C_c(E)$ and $c_g$ like before so that $[D, g(c)]=c_gg(\chi)$. We need to show
\[
\int_{g \in G} g(c)Tg(c)d{\mu_G}(g) - T \in K(H).
\]
By similar reasoning as above, this boils down to showing
\[
\int_{g \in G}g(c)\int^{\infty}_0\left((D\pm\sqrt{1+\lambda^2}i)^{-1} (c_gg(\chi)) (D\pm\sqrt{1+\lambda^2}i)^{-1}\right)d\lambda d{\mu_G}(g)
\]
is a compact operator. Notice that for each $g$ in $G$, the operator
\[
\int^{\infty}_0g(c)\left((D\pm\sqrt{1+\lambda^2}i)^{-1} (c_gg(\chi)) (D\pm\sqrt{1+\lambda^2}i)^{-1}\right)d\lambda
\]
is compact, so it suffices to show that the integral above is norm convergent over $G$, that is to show that for any $\epsilon>0$, there is a compact subset $K$ of $G$ such that for any compact subset $K'$ of $G\backslash K$, we have
\[
||\int_{K'} \int^{\infty}_0g(c)\left((D\pm\sqrt{1+\lambda^2}i)^{-1} (c_gg(\chi)) (D\pm\sqrt{1+\lambda^2}i)^{-1}\right)d\lambda d\mu_G(g)|| \leq \epsilon
\] 
This follows from that we have
\[
\sup_{K'\subset G\backslash K}||\int_{K'}g(c)\left((D\pm\sqrt{1+\lambda^2}i)^{-1} (c_gg(\chi))\right) d\mu_G(g)|| \leq (\sqrt{1+\lambda^2})^{-\frac{1}{2}}C_K
\]
where $C_K$ are constants which converge to $0$ as $K$ increases. Checking the last claim is left to the reader (Hint: combine Lemma \ref{lem_nondegen} and Lemma \ref{lem_nondegen_int}).
\end{proof}

\begin{example} \label{rem_npc} For any group $G$ which acts isometrically, properly and co-compactly on a simply connected, complete manifold $M$ with non-positive sectional curvature, the Witten perturbation $d_f + d_{f}^\ast$ of de-Rham operator on the Hilbert space $L^2(M, \Lambda^\ast T^\ast M)$ of the exterior algebra on $M$ (see \cite{Ka88}) satisfies the assumption in Theorem \ref{thm_gamma}, and thus defines a $G$-equivariant Kasparov cycle with property $(\gamma)$. Here, the operator $d_f=d+df\wedge$ is the sum of the exterior differential $d$ and the exterior multiplication by the one form $df$ of the function $f=d^2(x_0, x)$, the square of the distance function from some fixed base point $x_0$ of $M$. In particular, this gives us a concrete $G$-equivariant Kasparov cycle with property $(\gamma)$ for any co-compact closed subgroup $G$ of connected, semi-simple Lie group. We leave as an exercise for the reader to construct an unbounded cycle which satisfies the assumption in Theorem \ref{thm_gamma} for groups $G$ which act properly and co-compactly on a Euclidean building in the sense of \cite{KS91}.
\end{example}


\bibliography{Refs}

\begin{thebibliography}{BGHN19}

\bibitem[BCH94]{BCH94}
Paul Baum, Alain Connes, and Nigel Higson.
\newblock Classifying space for proper actions and {$K$}-theory of group
  {$C^\ast$}-algebras.
\newblock In {\em {$C^\ast$}-algebras: 1943--1993 ({S}an {A}ntonio, {TX},
  1993)}, volume 167 of {\em Contemp. Math.}, pages 240--291. Amer. Math. Soc.,
  Providence, RI, 1994.

\bibitem[BGH19]{BGH19}
J.~Brodzki, E.~Guentner, and N.~Higson.
\newblock A differential complex for cat(0) cubical spaces.
\newblock {\em Advances in Mathematics}, 347:1054 -- 1111, 2019.

\bibitem[BGHN19]{BGHN19}
J.~Brodzki, E.~Guentner, N.~Higson, and S.~Nishikawa.
\newblock On the {B}aum--{C}onnes conjecture and {CAT}(0)-cubical spaces.
\newblock In preparation, 2019.

\bibitem[Bla98]{Bl98}
Bruce Blackadar.
\newblock {\em {$K$}-theory for operator algebras}, volume~5 of {\em
  Mathematical Sciences Research Institute Publications}.
\newblock Cambridge University Press, Cambridge, second edition, 1998.

\bibitem[Ech17]{Ec17}
Siegfried Echterhoff.
\newblock {\em Bivariant KK-Theory and the Baum--Connes conjecure}, pages
  81--147.
\newblock Springer International Publishing, Cham, 2017.

\bibitem[GK02]{GK02}
Erik Guentner and Jerome Kaminker.
\newblock Exactness and the {N}ovikov conjecture.
\newblock {\em Topology}, 41(2):411--418, 2002.

\bibitem[Hig87]{Hig87}
Nigel Higson.
\newblock On a technical theorem of {K}asparov.
\newblock {\em J. Funct. Anal.}, 73(1):107--112, 1987.

\bibitem[HK97]{HK97}
Nigel Higson and Gennadi Kasparov.
\newblock Operator {$K$}-theory for groups which act properly and isometrically
  on {H}ilbert space.
\newblock {\em Electron. Res. Announc. Amer. Math. Soc.}, 3:131--142, 1997.

\bibitem[HK01]{HK01}
Nigel Higson and Gennadi Kasparov.
\newblock {$E$}-theory and {$KK$}-theory for groups which act properly and
  isometrically on {H}ilbert space.
\newblock {\em Invent. Math.}, 144(1):23--74, 2001.

\bibitem[Kas88]{Ka88}
G.~G. Kasparov.
\newblock Equivariant {$KK$}-theory and the {N}ovikov conjecture.
\newblock {\em Invent. Math.}, 91(1):147--201, 1988.

\bibitem[KS91]{KS91}
G.~G. Kasparov and G.~Skandalis.
\newblock Groups acting on buildings, operator {$K$}-theory, and {N}ovikov's
  conjecture.
\newblock {\em $K$-Theory}, 4(4):303--337, 1991.

\bibitem[KS03]{KS03}
Gennadi Kasparov and Georges Skandalis.
\newblock Groups acting properly on ``bolic'' spaces and the {N}ovikov
  conjecture.
\newblock {\em Ann. of Math. (2)}, 158(1):165--206, 2003.

\bibitem[Laf12]{La12}
Vincent Lafforgue.
\newblock La conjecture de {B}aum-{C}onnes \`a coefficients pour les groupes
  hyperboliques.
\newblock {\em J. Noncommut. Geom.}, 6(1):1--197, 2012.

\bibitem[Mey00]{Mey00}
Ralf Meyer.
\newblock Equivariant {K}asparov theory and generalized homomorphisms.
\newblock {\em $K$-Theory}, 21(3):201--228, 2000.

\bibitem[MN06]{MN06}
Ralf Meyer and Ryszard Nest.
\newblock The {B}aum-{C}onnes conjecture via localisation of categories.
\newblock {\em Topology}, 45(2):209--259, 2006.

\bibitem[NP19]{NP19}
S.~Nishikawa and V.~Proietti.
\newblock Groups with {S}panier--{W}hitehead duality.
\newblock In preparation, 2019.

\bibitem[Oza00]{Oza00}
Narutaka Ozawa.
\newblock Amenable actions and exactness for discrete groups.
\newblock {\em C. R. Acad. Sci. Paris S\'{e}r. I Math.}, 330(8):691--695, 2000.

\bibitem[STY02]{STY02}
G.~Skandalis, J.~L. Tu, and G.~Yu.
\newblock The coarse {B}aum-{C}onnes conjecture and groupoids.
\newblock {\em Topology}, 41(4):807--834, 2002.

\bibitem[Tu00]{Tu99}
Jean-Louis Tu.
\newblock The {B}aum-{C}onnes conjecture for groupoids.
\newblock In {\em {$C^*$}-algebras ({M}\"{u}nster, 1999)}, pages 227--242.
  Springer, Berlin, 2000.

\bibitem[Tu05]{Tu05}
Jean-Louis Tu.
\newblock The gamma element for groups which admit a uniform embedding into
  {H}ilbert space.
\newblock In {\em Recent advances in operator theory, operator algebras, and
  their applications}, volume 153 of {\em Oper. Theory Adv. Appl.}, pages
  271--286. Birkh\"{a}user, Basel, 2005.

\bibitem[Val02]{Va02}
Alain Valette.
\newblock {\em Introduction to the {B}aum-{C}onnes conjecture}.
\newblock Lectures in Mathematics ETH Z\"urich. Birkh\"auser Verlag, Basel,
  2002.
\newblock From notes taken by Indira Chatterji, With an appendix by Guido
  Mislin.

\end{thebibliography}
\bibliographystyle{alpha}

\end{document}